\documentclass[12pt]{amsart}
\usepackage{amsmath,amsfonts,amssymb,amsthm}
\theoremstyle{plain}
\newtheorem{theorem}{Theorem}[section]
\newtheorem{corollary}[theorem]{Corollary}
\newtheorem{lemma}[theorem]{Lemma}

\theoremstyle{definition}

\newtheorem{remark}[theorem]{Remark}

\begin{document}
\title{The Asymptotic Behaviour of the $m$-th Order Cardinal
$B$-Spline wavelet}

\author{
 Ronald Kerman, Mi-Ae Kim and Susanna Spektor}

\address{ Ronald Kerman, Brock University, Canada}
 \email{ rkerman@brocku.ca}

\address{ Mi-Ae Kim, Niagara College, Canada}
\email{ drkimmath@yahoo.com}

\address{ Susanna Spektor,
\noindent Address: Brock University, Canada}
\email{sanaspek@gmail.com}


\keywords{Cardinal $B$-spline wavelets, scaling function, exponential
decay, asymptotic behaviour.\\
This research was supported in part by NSERC grant A4021.}
\thanks{}
\begin{abstract}
It is well-known that the $m$-th order cardinal $B$-spline wavelet,
$\psi_{m},$ decays exponentially. Our aim in this paper is to
determine the exact rate of this decay and thereby to describe the asymptotic behaviour of $\psi_{m}$.
\end{abstract}

 \maketitle

\section{Introduction.}

Mallat and Meyer in [M, p. 225] define a wavelet, $\psi,$ with the
aid of a scaling function, $\varphi,$ in $L_{2}(R),$ which function
satisfies, among other things, a relation of the form
$$
\varphi(x)=\sum_{j \in Z} a_{j}\varphi(2x-j) \leqno (1.1)
$$
for certain scaling constants $a_{j};$ indeed,
$$
\psi(x):=\sum_{j \in Z}(-1)^j a_{1-j}\varphi(2x-j). \leqno (1.2)
$$

The cardinal $B$-spline scaling function, and hence the wavelets they determine, are given in terms of the
cardinal $B$-splines, $N_{m}.$ As in [C, p.17], the latter are defined inductively
by
$$
N_{1}(x)=\chi_{[0,1)}(x) {\mbox{     }} \mbox{{}and} {}{} \mbox{
${}N_{m}(x)=\int_{0}^{1}N_{m-1}(x-y)dy$},
$$
$m=2,3..$.
One has $N_{m} \in C^{m-2}(R)$ for $m\geq 2.$ Moreover, it is
supported in $[0,m]$ and is equal to a polynomial of degree $m-1$
on each interval of the form $[k, k+1), k=0,1,..,m-1.$

Though the family $\{N_{m}(x-j)\}_{j \in Z}$ is not an orthonormal
system, it is possible to construct a scaling function,
$\varphi_{m},$ from it so that $\{\varphi_{m}(x-j)\}_{j \in Z}$ is
such a system. Indeed, one can take
$$
\varphi_{m}(x):=\sum_{j \in Z} c_{j}N_{m}(x-j), \leqno (1.3)
$$
with
$$
c_{j}= \frac{1}{2 \pi}\int_{-
\pi}^{\pi}\frac{\cos{j\theta}}{\sqrt{P_{m}\left(\cos{\frac{\theta}{2}}\right)}}
d\theta, \leqno (1.4)
$$
in which $P_{m}$ is a polynomial of degree $m$ and
$$
P_{m}\left(\cos{\frac{\theta}{2}}\right)= 2 \pi \sum_{j \in Z
}|\widehat{N_{m}}(\theta +2 \pi j)|^2.
$$
(We use the convention
$\displaystyle{\widehat{f}(\lambda)=\frac{1}{\sqrt{2
\pi}}\int_{R}f(t)e^{-i \lambda t} dt, \lambda \in R},$ for the
Fourier transform.)

It is well-known that the $m$-th order cardinal $B$-spline wavelet
defined by $(1.2)$ in terms of the scaling constants of
$\varphi_{m}$ decays exponentially; see, for example, [D, corollary
5.4.2, pp. 150-152].

Our aim in this paper is to obtain the exact rate of the decay.
Its principal result is\\

 \textbf{Theorem A. } {\textit{The $m$-th order cardinal
$B$-spline wavelet, $\psi_{m}, m \geq 2,$ has the asymptotic form
$$
\psi_{m}(x)=\left[\sum_{x-m \leq j \leq x} (-1)^{r_j}E_j
\frac{e^{-\alpha_{0}r_j}}{\sqrt{r_j}} N_{m}(x-j)\right]
[1+o(1)], \leqno (1.5)
$$
as $|x|\rightarrow \infty,$ in which $\displaystyle{r_j=\left[\frac{|j|}{2}\right]}$ and $E_j$ depends only on the sign and parity of $j \in Z.$ The
constant $\alpha_{0}$ in (1.5) is given by
$$
\alpha_{0}=\ln \left[\frac{\sqrt{\mu_{m-1}+1}+\sqrt{\mu_{m-1}}}{\sqrt{\mu_{m-1}+1}-\sqrt{\mu_{m-1}}}\right],
$$
where
$$
\mu_{m-1}= \frac{(\lambda_{m-1}+1)^2}{4|\lambda_{m-1}|}
$$
and $\lambda_{m-1}$ is the $(m-1)$-st smallest negative root of
the Euler-Frobenius polynomial of degree $2m-2.$}\\

The constants $E_j$ are given explicitly for $j \in 2Z_{+}$ in the proof of Theorem 6.2. There, they are expressed in terms of constants $D_m=D_m(j)$ and $D_{m+1}=D_{m+1}(j)$, which are themselves specified in the proof of Theorem 5.1.

Estimates similar to (1.5) are given in [Ci, Theorem 1] for the Franklin functions.

To prove Theorem A we need a representation of $\psi_{m}$ similar
to the one for $\varphi_{m}$ in (1.3), namely,
$$
\psi_{m}(x)=\sum_{j \in Z} \gamma_{j}N_{m}(x-j).\leqno (1.6)
$$
As will be shown in Lemma 6.1 below,
$$
\gamma_{j}=(-1)^{j} \sum_{k \in Z}(-1)^k a_{k-j+1}c_{k},\leqno
(1.7)
$$
$j \in Z,$ where the $a_{j}$ appearing in (1.7) are the scaling
constants for $\varphi_{m}.$ These constants are given by the
formula
$$
a_{j}=2^{-m}\sum_{i=M-m}^{M+1}{m+1 \choose
M+1-i}\mathrel{\mathop{\sum_{l \in Z}}\limits_{l\equiv i}}
c_{\frac{l-i}{2}}b_{j-l},\leqno (1.8)
$$
for $j \in Z;$ here,
$$
b_{j}= \frac{1}{2 \pi}\int_{-
\pi}^{\pi}\sqrt{P_{m}\left(\cos{\frac{\theta}{2}}\right)} \cos{j
\theta} d\theta,
$$
$j \in Z,$ and $l\equiv i$ means $l=i\, mod\, 2.$ The constants $b_{j}$ come out of an equation inverse to (1.3), namely,
$$
N_{m}(x)=\sum_{j \in Z}b_{j} \varphi_{m}(x-j),\leqno
(1.9)
$$
$x \in R.$

The behaviour of $\psi_{m}(x)$ as $|x|\rightarrow \infty$ is
determined by that of $\gamma_{j}$ as $|j|\rightarrow \infty.$ To
describe the latter we need to know the long term behaviour of,
successively, the $c_{j}, b_{j}$ and $a_{j}.$ This is obtained in
sections 3,4 and 5, respectively. The proof of Theorem A is then
essentially given in section 6 by the determination of the
behaviour of $\gamma_{j}$ as $|j|\rightarrow \infty.$

We begin in the next section with a study of
$\displaystyle{\frac{1}{\sqrt{P_{m}\left(\cos{\frac{\theta}{2}}\right)}}}.$

\section{The function $ \displaystyle{\frac{1}{
\sqrt{P_{m}\left(\cos{\frac{\theta}{2}}\right)}}}.$}

It is shown in [Ch, p.90] that
$$
 P_{m}\left(\cos{\frac{\theta}{2}}\right)=\frac{1}{(2m-1)!}\prod_{k=1}^{m-1} \frac{1-2 \lambda_{k} \cos
\theta + \lambda_{k}^2}{|\lambda_{k}|},\leqno (2.1)
$$
in which $0>\lambda_{1}>..>\lambda_{m-1}$ are the first $m-1$
negative simple real roots of the Euler-Frobenius polynomial
$$
E_{2m-1}(z)=(2m-1)!z^{m-1}\sum_{k=-m+1}^{m-1} N_{2m}(m+ k)z^k,
$$
the remaining $m-1$ negative roots,
$\lambda_{1},..,\lambda_{2m-2},$ being such that $\lambda_{1}
\lambda_{2m-2}=...=\lambda_{m-1} \lambda_{m}=1$.\\

\begin{lemma} Let $P_{m}\left(\cos{\frac{\theta}{2}}\right)$ and $\lambda_{k}, \, k=1,...,m-1$, be as in (2.1).
 Then,
$$
\frac{1}{ \sqrt{P_{1}\left(\cos{\frac{\theta}{2}}\right)}}=
\frac{1}{\sqrt{\frac13+\frac23\, \cos^{2}\frac{\theta}{2}}},\leqno
(2.2)
$$
$$
\frac{1}{
\sqrt{P_{2}\left(\cos{\frac{\theta}{2}}\right)}}=\frac{1}{\sqrt{\frac{2}{15}+
\frac{11}{15}\, \cos^{2} \frac{\theta}{2}+ \frac{2}{15}\, \cos^{4}
\frac{\theta}{2}}},\leqno (2.3)
$$
and for $m\geq 3,$
$$
\frac{1}{ \sqrt{P_{m}\left(\cos{\frac{\theta}{2}}\right)}}= A
\left (1+\sum_{k=1}^{\infty}(-1)^k {- \frac12 \choose k}{ \left (
\mu_{m-1}+1 \right )^{-k}(1+B_{k}^{m-2}) }\sin^{2k}{\frac{\theta}{2}} \right ).\leqno (2.4)
$$
The positive constant $A$ is specified in the proof and $B_{k}^{m-2}$ is the final term in the finite recurrence
sequence
$$
B_{k}^{1}= \sum_{j=1}^{k}R(j,k) \left(
\frac{\mu_{2}+1}{\mu_{1}+1}\right)^j,
$$
where
$$
R(j,k)=\frac{{- \frac12 \choose j}{- \frac12 \choose k-j}}{{-
\frac12 \choose k}}
$$
and
$$
B_{k}^{l+1}= \sum_{j=1}^{k}R(j,k) \left(
\frac{\mu_{l+2}+1}{\mu_{l+1}+1}\right)^j(1+B_{j}^{l}),
$$
$l=1,...,m-1,$ with
$$
\mu_{i}= \frac{(\lambda_{i}+1)^2}{4|\lambda_{i}|},
$$
$i=1,...,m-1.$ Moreover, $B=\lim_{k \rightarrow \infty} B_{k}^{m-2}$ exists.
\end{lemma}

\begin{proof}
    The formulas (2.2) and (2.3) can be obtained directly from [Ch, p. 88, (4.2.10)].

    As for (2.4), we observe that, since $
1-2\lambda \cos\theta +\lambda^2=-4 \lambda \displaystyle{\left(x+
\frac{(\lambda+1)^2}{-4\lambda} \right)},$ with $x=\cos^2
\left(\frac{\theta}{2}\right),$ one has
\begin{align*}
\frac{1}{
\sqrt{P_{m}\left(\cos{\frac{\theta}{2}}\right)}}&=2^{-(m-1)}
\sqrt{(2m-1)!}
\left(\prod_{i=1}^{m-1}(x+\mu_{i})\right)^{-\frac12}\\
&=A\prod_{i=1}^{m-1}\left(1+\frac{x-1}{\mu_{i}+1}\right)^{-\frac12}\\
&=A\prod_{i=1}^{m-1}\left(1+\sum_{k=1}^{\infty}{-\frac12 \choose
k}\left(\frac{x-1}{\mu_{i}+1}\right)^k \right),
\end{align*}
where $A=2^{-(m-1)}\sqrt{(2m-1)!}
\,\prod_{i=1}^{m-1}(\mu_{i}+1)^{-\frac12}.$

Now,
$$
 \left[ 1+ \sum_{k=1}^{\infty}{ - \frac12 \choose k}{ \left(
\frac{x-1}{\mu_{1}+1} \right)^k} \right] \left[1+
\sum_{k=1}^{\infty}{ - \frac12 \choose k}{ \left(
\frac{x-1}{\mu_{2}+1} \right)^k} \right]
=1+\sum_{k=1}^{\infty}d_{k}(x-1)^{k},
$$
in which
\begin{align*}
d_{k}&={ - \frac12 \choose k} \left( \mu_{1}+1 \right)^{-k}+ {-
\frac12 \choose k}\left( \mu_{2}+1 \right)^{-k}+\sum_{j=1}^{k-1}{
- \frac12 \choose j}{- \frac12 \choose k-j} \left( \mu_{1}+1
\right)^{-j}\left( \mu_{2}+1\right)^{-(k-j)}\\
&={- \frac12 \choose k}\left( \mu_{2}+1 \right)^{-k}
\left[1+B_{k}^{1} \right]
\end{align*}
and
$$
B_{k}^{1}= \sum_{j=1}^{k}R(j,k) \left( \frac{\mu_{2}+1}{\mu_{1}+1}
\right)^j.
$$

We claim that  $\lim_{k \rightarrow \infty}{ B_{k}^{1}}$ exists.
Indeed, $\mu_{1}>\mu_{2}>...>\mu_{m-1}>0,$ so the claim will
follow from the Lebesgue dominated convergence theorem for
sequences once we show $R(j,k)$ is bounded independently of $j$
and $k,$ $1\leq j \leq k.$ But, on expressing the generalized
binomial coefficients in terms of gamma functions and using the
relation $\Gamma(x)\Gamma(1-x)= \displaystyle{\frac{\pi}{\sin \pi
x}},$ we obtain
$$
R(j,k)= \frac{1}{\sqrt{\pi}} \frac{\Gamma(k+1)\Gamma(j+ \frac12)\Gamma(k-j+
\frac12)}{\Gamma(k+ \frac12)\Gamma(j+1)\Gamma(k-j+1)}.
$$
Stirling's formula in the form $\Gamma(x) \sim
\sqrt{2\pi}e^{-x}x^{x- \frac12}$ then yields
\begin{align*}
R(j,k) &\sim  \sqrt{2e} \left( \frac{k+1}{k+ \frac12}\right)^k
\left( \frac{j+ \frac12}{j+1} \right)^j \left(\frac{k-j+
\frac12}{k-j+1} \right)^{k-j} \sqrt{ \frac{k+1}{(j+1)(k-j+1)}}\\
&\leq \sqrt{2e}\left(\frac{k+1}{k+\frac12}\right)^k\\
&=\sqrt{2e}\left(1+ \frac{1}{2k+1}\right)^k\\
&\leq \sqrt{2e^3}.
\end{align*}

Again,
$$
\left[1+ \sum_{k=1}^{\infty} {- \frac12 \choose k} \left(
\frac{x-1}{\mu_{2}+1}\right)^k \left[1+B_{k}^{1} \right] \right]
\left[1+ \sum_{k=1}^{\infty}{ - \frac12 \choose k}{ \left(
\frac{x-1}{\mu_{3}+1} \right)^k } \right] =1+\sum_{k=1}^{
\infty}h_{k}(x-1)^{k},
$$
with
$$
h_{k}={- \frac12 \choose k}(1+\mu_3)^{-k} (1+B_{k}^{2})
$$
and
$$
B_{k}^{2}=  \sum_{j=1}^{k} R(j,k) \left(
\frac{\mu_{3}+1}{\mu_{2}+1}\right)^j(1+B_{k}^1).
$$
An argument similar to the one involving the $B_{k}^{1}$ shows
$\lim_{k \rightarrow \infty}{ B_{k}^{2}}$ exists. Continuing like
this we finally get
\begin{align*}
\frac{1}{ \sqrt{P_{m}\left(\cos{\frac{\theta}{2}}\right)}}&=A
\left[1+\sum_{k=1}^{\infty}{-\frac12 \choose k} \left(
\frac{x-1}{\mu_{m-1}+1} \right)^k  (1+B_{k}^{m-2})\right]\\
&=A \left[1+\sum_{k=1}^{\infty}(-1)^k{-\frac12 \choose k} \left(
\mu_{m-1}+1 \right)^{-k} \sin^{2k}\left(\frac{\theta}{2}\right)
(1+B_{k}^{m-2})\right]
\end{align*} and $\lim_{k \rightarrow
\infty}{ B_{k}^{m-2}}$ exists, as asserted.\\
\end{proof}

\begin{corollary} Let $c_{j}$ be the Fourier coefficient in
(1.2).Then, for $j \gg 1,$
$$
c_{j}= (-1)^j A(1+B+o(1)) \sum_{k=j}^{\infty}(-1)^k{ - \frac12
\choose k}{2k \choose k-j} \left(4(\mu_{m-1}+1)\right)^{-k},
\leqno (2.5)
$$
where $A$ and $B$ are as in Lemma 2.1.\\
\end{corollary}

\begin{proof} According to Lemma 2.1,
$$
c_{j}= \frac {A}{2 \pi} \sum_{k=j}^{\infty}(-1)^k{ - \frac12
\choose k} \left(\mu_{m-1}+1\right)^{-k} (1+B_{k}^{m-2})
\int_{-\pi}^{\pi} \sin^{2k}\frac{\theta}{2} \cos j\theta
d\theta.\leqno (2.6)
$$
Now,
$$
\sin^{2k}\theta= \frac{1}{2^{2k}}{2k \choose k}+ \frac{1}{2^{2k-1}}\sum_{i=0}^{k-1}(-1)^{k-i}{ 2k
\choose i} \cos 2(k-i) \theta,
$$
so
$$
\int_{-\pi}^{\pi}\sin^{2k} \frac{ \theta}{2} \cos j\theta d\theta=
\frac{1}{2^{2k-1}}\sum_{i=0}^{k-1}(-1)^{k-i}{2k \choose
i}\int_{-\pi}^{\pi} \cos(k-i)\theta \cos j \theta d\theta
$$
$$
= \left\{
\begin{array}{rcl}
\displaystyle{0 , \,  {} k <j}\\
\displaystyle{(-1)^{j}\frac{\pi}{2^{2k}-1} {2k \choose k-j} , \, {} k\geq j.}\\
\end{array} \right.
$$
Substitution in (2.6) and the observation that
$B_{k}^{m-2}=B+o(1)$ for $k\gg 1$ yields (2.5).\\
\end{proof}

\section{The constants $c_{j}.$}

The purpose of this section is to prove\\

\begin{theorem} Let $c_{j}$ be given by (1.4). Then,
 $$
c_{j} \sim {(-1)^j}K_{c} \frac{e^{-\alpha_{0}|j|}}{
\sqrt{|j|}},\leqno (3.1)
$$
as $|j|\rightarrow\infty,$ where
$$
\alpha_{0}=\ln \left[\frac{\sqrt{\mu_{m-1}+1}+\sqrt{\mu_{m-1}}}{\sqrt{\mu_{m-1}+1}-\sqrt{\mu_{m-1}}}\right]
$$
and
$$
\displaystyle{K_{c}=\frac {1}{\sqrt\pi}\,
A(1+B)\left(1+\frac{1}{\mu_{m-1}}\right)^{\frac14
}}
$$
\end{theorem}

\begin{proof}
 According to Corollary 2.2,

$$
c_{j}=c_{|j|}=(-1)^j A(1+B+o(1))\sum_{k=|j|}^{\infty}(-1)^k{-
\frac12 \choose k}{2k \choose k-|j|}
\left(4(\mu_{m-1}+1)\right)^{-k},
$$
for $|j|\gg 1.$ We have, when $j>0$
\begin{align*}
\sum_{k=j}^{\infty}(-1)^k{- \frac12 \choose k}{2k \choose k-j}
\left(4(\mu_{m-1}+1)\right)^{-k}
= \frac{1}{\sqrt\pi}\sum_{k=j}^{\infty} \frac{\Gamma(k+
\frac12)\Gamma(2k+1)}{\Gamma(k+1)\Gamma(k+j+1)\Gamma(k-j+1)}
\left(4(\mu_{m-1}+1)\right)^{-k}
\end{align*}

The ratio of the $k+1$-st term to the $k$-th term in the last series is
$$
r(n)=\frac{1}{4(\mu_{m-1}+1)}\frac{(2n-1)^2}{(n^2-j^2)}, {} n=k+1.
$$
As
$$
r^{'}(n)= \frac{2(2n-1)(n-2j)^2}{4(\mu_{m-1}+1)(n^2-j^2)^2},
$$
we conclude
$$
\lim_{n\rightarrow \infty}{r(n)} = \displaystyle{
\frac{1}{\mu_{m-1} +1}}; \leqno(i)
$$
$$
r(n) \mbox{} \, \mbox{decreses until $n=2j^2$, after which it increases};\leqno(ii)
$$
$$
r(n)=1 \mbox{} \, \mbox{when}\leqno(iii)
$$
$$
(4c-1)n^2-4cn+c+j^2=0, \mbox{} c=\frac{1}{4(\mu_{m-1}+1)}
$$
or
$$
n=\frac{4c-\sqrt{4c+4j^2(1-4c)}}{2(4c-1)}\backsim \frac{j}{\sqrt{1-4c}}=j \sqrt{1+ \frac{1}{\mu_{m-1}}}.
$$

Take $k=l+j$, so that
$$
c_{j}=(-1)^j\frac{A}{\sqrt\pi}(1+B+o(1))\sum_{l=0}^{\infty}\frac{\Gamma{l+j+ \frac12}\Gamma(2l+2j+1)}{\Gamma(l+j+1)\Gamma(l+2j+1)\Gamma(l+1)}(4(\mu_{m-1}+1))^{-(l+j)}
$$

Let
$$
l_{1}=\left( \sqrt{1+ \frac{1}{\mu_{m-1}}}-1\right)j=\alpha  j.
$$
Then, according to [L, p. 274], we have, for $l=l_{1}+h$ and $|h|\leq j^{\frac35}$,
$$
\Gamma(l_1+j+\frac12+h)= \Gamma(l_1+j+\frac12)(l_1+j-\frac12)^h exp \left(\frac{h^2}{2(l_1+j-\frac12)}\right)[1+O(j^{-\frac15})]
$$
$$
\Gamma(2l_1+2j+1+2h)= \Gamma(2l_1+2j+1)(2l_1+2j)^{2h} exp \left(\frac{4h^2}{2(2l_1+2j)}\right)[1+O(j^{-\frac15})]
$$
$$
\Gamma(l_1+j+1+h)= \Gamma(l_1+j+1)(l_1+j)^{h} exp \left(\frac{h^2}{2(l_1+j)}\right)[1+O(j^{-\frac15})]
$$
$$
\Gamma(l_1+2j+1+h)= \Gamma(l_1+2j+1)(l_1+2j)^{h} exp \left(\frac{h^2}{2(l_1+2j)}\right)[1+O(j^{-\frac15})]
$$
and
$$
\Gamma(l_1+1+h)= \Gamma(l_1+1)l_1^{h} exp \left(\frac{h^2}{2l_1}\right)[1+O(j^{-\frac15})]
$$
Thus, with $l=l_1+h,$ we have, uniformly in h, $|h|\leq j^{\frac35},$
$$
\frac{\Gamma(l+j+ \frac12)\Gamma(2l+2j+1)}{\Gamma(l+j+1)\Gamma(l+2j+1)\Gamma(l+1)}
$$
equal to
$$
\frac{\Gamma(l_1+j+ \frac12)\Gamma(2l_1+2j+1)}{\Gamma(l_1+j+1)\Gamma(l_1+2j+1)\Gamma(l_1+1)}
$$
times
$$
exp \left[-\frac{h^2}{2} \left( -\frac{1}{l_1+j- \frac12}-\frac{4}{2l_1+2j}+\frac{1}{l_1+j}+\frac{1}{l_1+2j}+\frac{1}{l_1}\right)\right]
$$
times
$$
\left[\frac{(l_1+j-\frac12)(2l_1+2j)^2}{(l_1+j)(l_1+2j)l_1}\right]^h\sim \left[\frac{4(l_1+j)^3}{(l_1+j)(l_1+2j)l_1}\right]^h=\left[\frac{4(l_1+j)^2}{l_1(l_1+2j)}\right]^h.
$$

Now,
$$
\frac{4(l_1+j)^2}{l_1(l_1+2j)}=4(\mu_{m-1}+1)
$$
and
$$
j \left[ -\frac{1}{l_1+j-\frac12}-\frac{4}{2l_1+2j}+\frac{1}{l_1+j}+\frac{1}{l_1+2j}+\frac{1}{l_1}\right]
$$
$$
\longrightarrow -\frac{2}{\alpha +1}+\frac{1}{\alpha+2}+\frac{1}{\alpha}=\frac{2}{\alpha(\alpha+1)(\alpha+2)}=\beta.
$$

Setting
$$
d_{l}=\left[\frac{\Gamma(l+j+\frac12)\Gamma(2l+2j+1)}{\Gamma(l+j+1)\Gamma(l+2j+1)\Gamma(l+1)}\right] \left(4(\mu_{m-1}+1)\right)^{-(l+j)}
$$
we have shown
\begin{align*}
\sum_{|l-l_1|\leq j^{\frac35}} d_l \sim \sqrt{j} d_{l_1} \sum_{|l-l_{1}|\leq j^{\frac35}}\frac{e^{-\frac{h^2}{2}\beta}}{\sqrt{j}} \\
&\sim \sqrt{j}d_{l_1} \int_{-j^{\frac{1}{10}}}^{j^{\frac{1}{10}}}e^{- \frac{-\beta t^2}{2}}dt, {}{}\, \mbox{{}\, by [L, p. 275],}\\
&\sim \sqrt{j} d_{l_1} \sqrt{\frac{2}{ \beta}} \int_{-\infty}^{\infty} e^{-t^2} dt \\
&\sim \sqrt{\frac{2 \pi j}{\beta}}d_{l_1},
\end{align*}
as $j\rightarrow \infty$

Again, when $|h|>j^{\frac35}$ there exists $\rho,$ $0< \rho<1,$ such that
$$
d_l\leq\rho^{l_1-j^{\frac35}-l}d_{l_1-j^{\frac35}}, {} \, \mbox{{} \, where \, {} $0 \leq l \leq l_1+j^{\frac35}$},
$$
and
$$
d_l \leq \rho^{l-l_1-j^{\frac35}}d_{l_1+j^{\frac35}}, {} \, \mbox{{} \, where \, {} $l_1+j^{\frac35} \leq l$},
$$
so
$$
\sum_{l=0}^{l_1-j^{\frac35}}d_{l} \leq d_{l_1-j^{\frac35}}\sum_{l=0}^{l_1-j {\frac35}}\rho^{l_1-l-j^{\frac35}} =d_{{l_1}-j^{\frac35}}\sum_{l=0}^{l_1-j^{\frac35}} \rho^l \leq \frac{d_{l_1-j^{\frac35}}}{1-\gamma}
$$ and, similarly,
$$
\sum_{l=l_1+j^{\frac35}}^{\infty}d_l \leq \frac{d_{l_1+j^{\frac35}}}{1- \rho}.
$$

Altogether, then,
$$
c_{j} \sim (-1)^j A(1+B) \sqrt{\frac{2j}{\beta}}d_{l_1}.
$$

Next, Stirling's formula (in the form $\Gamma(x) \sim \sqrt{2 \pi} e^{-x}x^{x- \frac12}$) yields
$$
(4(\mu_{m-1}+1))^{l_1+j}jd_{l_1} \sim \frac{j}{\sqrt{2 \pi}}{} \exp [-(l_1+j+\frac12+2l_1+2j+1-l_1-j-1-l_1-2j-1-l_1-1)]
$$
times
$$
\frac{(l_1+j+ \frac12)^{l_1+j}(2l_1+2j+1)^{2l_1+2j+ \frac12}}{(l_1+j+1)^{l_1+j+ \frac12}(l_1+2j+1)^{l_1+2j+ \frac12}(l_1+1)^{l_1+\frac12}}
$$
which equals
\begin{align*}
\frac{e^{\frac32}}{\sqrt{2 \pi}}\frac{\left(1+\frac{\frac12}{l_1+j} \right)^{l_1+j}}{\left(1+\frac{1}{l_1+j}\right)^{l_1+j}}
\frac{(2l_1+2j+1)^{2l_1+2j+\frac12}}{(l_1+2j+1)^{l_1+2j+\frac12}(l_1+1)^{l_1}} \sqrt{\frac{j^2}{(l_1+j+1)(l_1+1)}}\\
\sim \frac{e}{\sqrt{2 \pi \alpha (\alpha+1)}}4^{l_1+j+\frac14}\frac{(l_1+j+\frac12)^{2(l_1+j)+\frac12}}{(l_1+j+j+1)^{l_1+2j+\frac12}(l_1+j-j+1)^{l_1}}\\
\sim \frac{e}{\sqrt{2 \pi \alpha (\alpha+1)}}4^{l_1+j+\frac14}\frac{\left(1+\frac{\frac12}{l_1+j}\right)^{2(l_1+j)}}{\left(1+\frac{j}{l_1+j}+\frac{1}{l_1+j}\right)^{l_1+2j} \left(1- \frac{j}{l_1+j}+\frac{1}{l_1+j}\right)^{l_1}}
\end{align*}
times
\begin{align*}
\left(\frac{l_1+j+\frac12}{l_1+2j+1}\right)^{\frac12}
& \sim \frac{4^{l_1+j+\frac14}e^2}{\sqrt{2 \pi \alpha(\alpha+1)}} \left(\frac{\alpha+1}{\alpha+2}\right)^{\frac12} \frac{1}{\left(1+\frac{1}{\alpha+1}+\frac{1}{l_1+j}\right)^{l_1+2j}}\frac{1}{\left(1-\frac{1}{\alpha+1}+\frac{1}{l_1+j}\right)^{l_1}}\\
& \sim \frac{4^{l_1+j}e^2}{\sqrt{\pi \alpha(\alpha+2)}}(1-\gamma^2)^{-(l_1+j)}\left(\frac{1- \gamma}{1+ \gamma}\right)^j \frac{1}{\left(1+\frac{1}{(1+ \gamma)(l_1+j)}\right)^{l_1+j}} \frac{1}{\left(1+\frac{1}{(1- \gamma)(l_1+j)}\right)^{l_1+j}}
\end{align*}
times
\begin{align*}
&\left(1+\frac{\gamma}{(1- \gamma)j}\right)^{j} \left(1+\frac{\gamma}{(1+\gamma)j}\right)^{-j} \, {} \,
\left(\gamma=\frac{1}{1+ \alpha}, \, {} l_1+j=(\alpha+1)j=\frac{j}{\gamma}\right)\\
&\sim \frac{(4(\mu_{m-1}+1))^{l_1+j}e^2}{\sqrt{\pi \alpha(\alpha+2)}}e^{-\alpha_0j}\exp\left[\frac{\gamma}{1- \gamma}-\frac{\gamma}{1+ \gamma}-\frac{1}{1- \gamma}-\frac{1}{1+ \gamma}\right]\\
&=\frac{(4(\mu_{m-1}+1))^{l_1+j}}{\sqrt{\pi \alpha(\alpha+2)}}e^{-\alpha_0j}
\end{align*}
Finally, then,
$$
c_j \sim(-1)^j K_c \frac{e^{- \alpha_0 j}}{\sqrt{j}},
$$
as $j\rightarrow \infty,$ where
$$K_{c}=A(1+B)\sqrt{\frac{\alpha+1}{\pi}}=\frac{A(1+B)}{\sqrt{\pi}}\left(1+\frac{1}{\mu_{m-1}}\right)^{\frac14}.
$$
\end{proof}

\section{The constants $b_{j}.$}

Using the methods of section 3 one can prove
\\

\begin{theorem} Suppose $A, \alpha_0$ and $\mu_{m-1}$ are as in section 3 and let
$$
b_{j}= \frac{1}{2 \pi}\int_{-\pi}^{\pi}\sqrt{P_m \left(\cos{\frac{\theta}{2}}\right)} \cos{j\theta}d\theta,\leqno (4.1)
$$
$j \in Z$. Then,
$$
b_{j}\sim(-1)^{j+1}K_b\frac{e^{-\alpha_0|j|}}{|j|^{\frac32}},
$$
with
$$
K_{b}=\frac{1}{\sqrt{\pi}}\frac{1+C}{A}\left(1+\frac{1}{\mu_{m-1}}\right)^{-\frac14}
$$
and
$$
C=\lim_{k\rightarrow \infty}C_{k}^{m-2},
$$
where $C_{k}^{m-2}$ is the last term in the finite recurrence sequence
$$
C_{k}^{1}= \sum_{i=1}^{k}S(i,k)\left(\frac{\mu_{2}+1}{\mu_{1}+1}\right)^j,
$$
and
$$
C_{k}^{n+1}= \sum_{i=1}^{k}S(i,k)\left(\frac{\mu_{n+2}+1}{\mu_{n+1}+1}\right)^k(1+C_{i}^n), \, {} n=1,..,m-3;
$$
here,
$$
\displaystyle{S(i,k)=\frac{{\frac12 \choose i}{\frac12 \choose k-i}}{{\frac12 \choose k}}}.
$$
\,
\end{theorem}

We only remark that the key difference lies in the $d_{l}$ of section 3 being replaced by

\begin{align*}
\frac{-\Gamma(l+j-\frac12)(\Gamma(2l+2j+1))}{\Gamma(l+j+1)\Gamma(l+2j+1)\Gamma(l+1)}(4(\mu_{m-1}+1))^{-(l+j)}
=-\frac{d_{l}}{l+j-\frac12}.
\end{align*}

$$
$$

\section{The scaling constants $a_{j}.$\\}

The purpose of this section is to prove

\begin{theorem} Let $a_{j}$ be the $j$-th scaling constant
of the function $\varphi_{m},$ given by (1.8). Then, with
$\alpha_{0}$ as in Theorem A and $r=\left[\frac{|j-m|}{2}\right]$,
one has
$$
a_{j}\sim (-1)^{r} D_j \frac{e^{-\alpha_{0}r}}{\sqrt{r}}\leqno (5.1)
$$
as $\displaystyle{|j|\rightarrow \infty}.$ Here, $D_j>0$ depends only on
the sign and parity of $j.$
\end{theorem}

 To do this we require

\begin{lemma} Set $M=\left[\frac{m}{2}\right].$ Then, the
scaling constant
$$
a_{j}=2^{-m}\sum_{i=M-m}^{M+1} {m+1 \choose
M+1-i}\mathrel{\mathop{\sum_{k \in Z}}\limits_{k\equiv i}}
c_{\frac{k-i}{2}} b_{j-k}, \leqno (5.2)
$$
$j \in Z;$ as before, $k\equiv i$ means $k=i \, mod \, 2.$
\end{lemma}
\begin{proof} According to [D, p. 148] the cardinal $B$ -spline,
$N_{m},$ satisfies the scaling relation
\begin{align*}
N_{m}(x)&=2^{-m}\sum_{i=0}^{m+1} {m+1 \choose i}
N_{m}(2x-M-1+i)\\
&=2^{-m}\sum_{i=M-m}^{M+1} {m+1 \choose M+1-i} N_{m}(2x-i)
\end{align*}

From formulas (1.3) and (1.9) we then have
\begin{align*}
\phi_{m}(x)&=\sum_{k \in Z}c_{k}N_{m}(x-k)=\sum_{k \in
Z}c_{k}2^{-m}\sum_{i=M-m}^{M+1}{m+1 \choose M+1-i}N_{m}(2x-2k-i)\\
& =2^{-m}\sum_{i=M-m}^{M+1}{m+1 \choose M+1-i}\sum_{k \in
Z}c_{k}N_{m}(2x-2k-i)\\
& =2^{-m}\sum_{i=M-m}^{M+1}{m+1 \choose M+1-i}\mathrel{\mathop{\sum_{k \in Z}}\limits_{k\equiv i}}c_{\frac{k-i}{2}}N_{m}(2x-k)\\
& =2^{-m}\sum_{i=M-m}^{M+1}{m+1 \choose
M+1-i}\mathrel{\mathop{\sum_{k \in Z}}\limits_{k\equiv
i}}c_{\frac{k-i}{2}}\sum_{l \in
Z}b_{l}\varphi_{m}(2x-k-l)\\
& =2^{-m}\sum_{i=M-m}^{M+1}{m+1 \choose
M+1-i}\mathrel{\mathop{\sum_{k \in Z}}\limits_{k\equiv
i}}c_{\frac{k-i}{2}}\sum_{j \in
Z}b_{j-k}\varphi_{m}(2x-j)\\
& =\sum_{j \in Z}\left[2^{-m}\sum_{i=M-m}^{M+1}{m+1 \choose
M+1-i}\mathrel{\mathop{\sum_{k \in Z}}\limits_{k\equiv
i}}c_{\frac{k-i}{2}}b_{j-k}\right]\varphi_{m}(2x-j)
\end{align*}

Thus, (5.2) holds, in view of the scaling relation (1.1).\\
\end{proof}
$$
$$
\textit{Proof of Theorem 5.1.} For $j>>1.$ Fix $i$ in (5.2), say
$i=m+2n_{i},$ and consider
\begin{align*}
\mathrel{\mathop{\sum_{k \in Z}}\limits_{k\equiv i}}
c_{\frac{k-i}{2}}b_{j-k}&=\mathrel{\mathop{\sum_{k \in Z}}\limits_{k\equiv m}}c_{\frac{k-i}{2}}b_{j-k}\\
 & =\sum_{n \in
Z}c_{n-n_{i}}b_{j-m-2n}.
\end{align*}

To begin,
\begin{align*}
\left|\sum_{n =-\infty}^0
c_{n-n_{i}}b_{j-m-2n}\right|&=\left|\sum_{n =0}^{\infty}
c_{n+n_{i}}b_{j-m+2n}\right|\\
& \leq K
\sum_{n=0}^{\infty}|c_{n+n_{i}}|\frac{e^{-\alpha_{0}(j-m+2n)}}{(j-m+2n)^{\frac32}}\\
& \leq K \frac{e^{-\alpha_{0}(j-m)}}{(j-m)^{\frac32}}\sum_{n=0}^{\infty}e^{-2\alpha_{0}n}\\
& =o\left(\frac{e^{-\alpha_{0}r}}{\sqrt r}\right).
\end{align*}

Next,
\begin{align*}
\left|\sum_{n =1}^{r-[\sqrt r]} c_{n-n_{i}}b_{j-m-2n}\right|\\
 &\leq
K \sum_{n=1}^{r-[\sqrt r]}\frac{e^{-\alpha_{0}n}}{\sqrt
n}\frac{e^{-\alpha_{0}(j-m-2n)}}{(j-m-2n)^{\frac32}}\\
& \leq K e^{-\alpha_{0}(j-m)} \sum_{n=1}^{r-[\sqrt r]}\frac{1}{\sqrt
n}\frac{e^{\alpha_{0}n}}{(j-m-2n)^{\frac12}}\\
& \leq K e^{-\alpha_0}e^{-\alpha_0[\sqrt{r}]}\sum_{n=1}^{r-[\sqrt
r]}\frac{1}{\sqrt{n}(r-n)^{\frac32}}\\
& =o\left(\frac{e^{-\alpha_{0}r}}{\sqrt r}\right).
\end{align*}

Again,
\begin{align*}
\sum_{n=r-[\sqrt r]+1}^{\infty}c_{n-n_{i}}b_{j-m-2n}\\
& \sim
\sum_{n=r-[\sqrt
r]+1}^{\infty}(-1)^{n-n_{i}}K_{c}\frac{e^{-\alpha_{0}(n-n_{i})}}{\sqrt{n-n_{i}}}b_{j-m-2n}\\
& =K_{c}\frac{e^{-\alpha_{0}r}}{\sqrt r}\sum_{n=r-[\sqrt
r]+1}^{\infty}(-1)^{n-n_{i}}e^{-\alpha_{0}(n-n_{i}-r)}\sqrt{\frac{r}{n-n_{i}}}b_{j-m-2n}\\
& =K_{c}\frac{e^{-\alpha_{0}r}}{\sqrt r}\sum_{k=-n_{i}-[\sqrt{
r}]-1}^{\infty}(-1)^{k+r}e^{-\alpha_{0}k}\sqrt{\frac{r}{r+k}}b_{j-m-2k-2n_{i}-2r}\\
& \sim (-1)^r K_{i}^{m}\frac{e^{-\alpha_{0}r}}{\sqrt r},
\end{align*}
where
$$
K_{i}^{m}=K_{c}\sum_{k \in
Z}(-1)^{k}e^{-\alpha_{0}k}b_{2k+2r-(j-m)+2n_{i}}.
$$

When $i=m+1+2n_{i}$ we arrive at
$$
\mathrel{\mathop{\sum_{k \in Z}}\limits_{k\equiv m+1}}
c_{\frac{k-i}{2}}b_{j-k}\sim(-1)^r
K_{i}^{m+1}\frac{e^{-\alpha_{0}r}}{\sqrt r},
$$
in which
$$
K_{i}^{m+1}=K_{c}\sum_{k \in
Z}(-1)^{k}e^{-\alpha_{0}k}b_{2k+2r-(j-m-1)+2n_{i}}.
$$

Altogether, then, (5.1) holds, with
$$
D_j=D_{m}+D_{m+1}=2^{-m}\mathrel{\mathop{\sum^{M+1}_{i=M-m}}\limits_{i\equiv m}}
{m+1 \choose
M+1-i}K_{i}^{m}+2^{-m}\mathrel{\mathop{\sum^{M+1}_{i=M-m}}\limits_{i\equiv
m+1}}{m+1 \choose M+1-i}K_{i}^{m+1}.\,\, \Box \\
$$

\textbf{Remarks 5.3.\\}

1. $D_j$ is a constant over all $j>0$ having the same parity. Thus,
if $j\equiv m,$ one has $2r=j-m,$ while if $j\equiv m+1,
2r=j-m-1.\\$

2. When $j<<-1, r= \left[\frac{-j+m}{2}\right],$
$$
K_{i}^{m}=K_{c}\sum_{k \in
Z}(-1)^{k}e^{-\alpha_{0}k}b_{2k+2r+j-m+2n_{i}} \leqno (5.3)
$$
and
$$
K_{i}^{m+1}=K_{c}\sum_{k \in
Z}(-1)^{k}e^{-\alpha_{0}k}b_{2k+2r+j-m-1+2n_{i}}.\leqno (5.4)
$$
$$
$$

\section{The constants $\gamma_{j}.$\\}

 We here study the asymptotic behaviour of the constants $\gamma_{j}$ in (1.6). A
formula for them is given in\\

\begin{lemma} The constants $\gamma_{j}$ are given in terms
of the constants $a_{j}$ and $c_{j}$ by

$$
\gamma_{j}=(-1)^j \sum_{k \in Z}(-1)^k a_{k-j+1}c_{k}, \leqno
(6.1)
$$
$j \in Z$.\\
\end{lemma}

\begin{proof} Using the formulas (1.2) and (1.3) for $\psi_{m},$ and
$\varphi_{m}$ respectively, one obtains
\begin{align*}
\psi_{m}(x)&=\sum_{l \in Z} (-1)^l a_{1-l} \varphi_{m}(2x-l)\\
&=\sum_{l \in Z} (-1)^l a_{1-l} \sum_{n \in
Z}c_{n}N_{m}(2x-l-n)\\
& =\sum_{l \in Z} (-1)^l a_{1-l} \sum_{j \in
Z}c_{j-l}N_{m}(2x-j)\\
& =\sum_{j \in Z} \left[\sum_{l \in Z}(-1)^l
a_{1-l}c_{j-l}\right]N_{m}(2x-j)\\
& =\sum_{j \in Z} \left[\sum_{k \in Z}(-1)^{j-k}
a_{k-j+1}c_{k}\right]N_{m}(2x-j),
\end{align*}
which proves (6.1).\\
\end{proof}

\begin{theorem} Let $\gamma_{j}$ be given by (6.1) and suppose
$\alpha_{0}$ is as in Theorem A.
 Then, with
$\displaystyle{r=\left[\frac{|j|+1}{2}\right]},$
$$
\gamma_{j} \sim (-1)^r E_j \frac{e^{-\alpha_{0}r}}{\sqrt r},
$$
as $|j|\rightarrow \infty,$ in which $E_j$ depends only on the sign and
parity of $j.\\$
\end{theorem}
\begin{proof} For $j>>1$ and $j\equiv 0.$
We write
\begin{align*}
\gamma_{j}&=\left(\sum_{k=-\infty}^{-\frac{j}{2}-1}+\sum_{k=-\frac{j}{2}}^{\frac{j}{2}}+\sum_{k=\frac{j}{2}+1}^{j-1}+\sum_{k=j}^{\infty}\right)(-1)^k
a_{k-j+1}c_{k}\\
& =S_{1}+S_{2}+S_{3}+S_{4}.
\end{align*}

Consider $S_{2}$ first. One has
\begin{align*}
S_{2}&=\sum_{k=-\frac{j}{2}}^{\frac{j}{2}}(-1)^ka_{k-j+1} c_{k}\\
&\sim
\mathrel{\mathop{\sum^{\frac{j}{2}}_{k=-\frac{j}{2}}}\limits_{k-j\equiv
m+1}}(-1)^k(-1)^{\frac{j-k+m-1}{2}}D_{m}\frac{e^{-\alpha_{0}(\frac{j-k+m-1}{2})}}{\sqrt{\frac{j-k+m-1}{2}}}c_{k}\\
&
+\mathrel{\mathop{\sum^{\frac{j}{2}}_{k=-\frac{j}{2}}}\limits_{k-j\equiv
m}}(-1)^k(-1)^{\frac{j-k+m}{2}}D_{m+1}\frac{e^{-\alpha_{0}(\frac{j-k+m}{2})}}{\sqrt{\frac{j-k+m}{2}}}c_{k}\\
& =(-1)^{\frac{j}{2}}\frac{e^{-\alpha_{0}\frac{j}{2}}}{\sqrt{
\frac{j}{2}}}[\mathrel{\mathop{\sum^{\frac{j}{2}}_{k=-\frac{j}{2}}}\limits_{k-j\equiv
m+1}}(-1)^{\frac{k+m+1}{2}}D_{m}\sqrt{\frac{\frac{j}{2}}{\frac{j}{2}-
\frac{k-m+1}{2}}}e^{-\alpha_{0}(\frac{m-k-1}{2})}c_{k}\\
& +
\mathrel{\mathop{\sum^{\frac{j}{2}}_{k=-\frac{j}{2}}}\limits_{k-j\equiv
m}}(-1)^{\frac{k+m}{2}}D_{m+1}\sqrt{\frac{\frac{j}{2}}{\frac{j}{2}-
\frac{k-m}{2}}}e^{-\alpha_{0}(\frac{m-k}{2})}c_{k}]\\
& \sim (-1)^{\frac{j}{2}}E_j \frac{e^{-\alpha_{0}\frac{j}{2}}}{\sqrt
{\frac{j}{2}}},
\end{align*}
with
\begin{align*}
E_j=\mathrel{\mathop{\sum_{k \in Z}}\limits_{k-j\equiv m+1}}
(-1)^{\frac{k+m+1}{2}}D_{m}e^{-\alpha_{0}(\frac{m-k-1}{2})}c_{k}+\mathrel{\mathop{\sum_{k
\in Z}}\limits_{k-j\equiv
m}}(-1)^{\frac{k+m}{2}}D_{m+1}e^{-\alpha_{0}(\frac{m-k}{2})}c_{k}\\
\sim K_c e^{-\alpha_0 \frac{(m-1)}{2}} \mathrel{\mathop{\sum_{k \in Z}}\limits_{k-j\equiv
m}}(-1)^{\frac{k+m+1}{2}-|k|}D_{m}e^{-\alpha_0(|k|-\frac{k}{2})}+e^{-\alpha_{0}\frac{m}{2}}\mathrel{\mathop{\sum_{k \in Z}}\limits_{k-j\equiv
m}}(-1)^{\frac{k+m}{2}-|k|}D_{m+1}e^{-\alpha_0(|k|-\frac{k}{2})}.
\end{align*}

Next,
\begin{align*}
|S_{1}|&=\left|\sum_{k=\frac{j}{2}+1}^{\infty}(-1)^{k}a_{j+k-1}c_{k}\right|\\
& \leq K
\left[\mathrel{\mathop{\sum^\infty_{k=\frac{j}{2}+1}}\limits_{j+k\equiv
m+1}}\frac{e^{-\alpha_{0}(\frac{j+k-1-m}{2})}}{\sqrt{\frac{j+k-1-m}{2}}}\frac{e^{-\alpha_{0}k}}{\sqrt
k}
+\mathrel{\mathop{\sum^\infty_{k=\frac{j}{2}+1}}\limits_{j+k\equiv
m}}\frac{e^{-\alpha_{0}(\frac{j+k-m}{2})}}{\sqrt{\frac{j+k-m}{2}}}\frac{e^{-\alpha_{0}k}}{\sqrt
k}\right]\\
& \leq K \frac{e^{-\alpha_{0}\frac{j}{2}}}{\sqrt \frac{j}{2}}\left[
\mathrel{\mathop{\sum^\infty_{k=\frac{j}{2}+1}}\limits_{j+k\equiv
m+1}}{\sqrt{\frac{\frac{j}{2}}{{\frac{j}{2}+\frac{k-m-1}{2}}}}}{e^{-\alpha_{0}(\frac{3k}{2}-\frac{(m+1)}{2})}}
+\mathrel{\mathop{\sum^\infty_{k=\frac{j}{2}+1}}\limits_{j+k\equiv
m}}{\sqrt{\frac{\frac{j}{2}}{{\frac{j}{2}+\frac{k-m}{2}}}}}{e^{-\alpha_{0}(\frac{3k}{2}-\frac{m}{2})}}
\right]\\
& =o \left(\frac{e^{-\alpha_{0}\frac{j}{2}}}{\sqrt
\frac{j}{2}}\right).
\end{align*}

Again,
\begin{align*}
|S_{3}|&=\left|\sum_{k=\frac{j}{2}+1}^{j-1}(-1)^{k}a_{k-j+1}c_{k}\right|\\
&\leq K \left[
\mathrel{\mathop{\sum^{j-1}_{k=\frac{j}{2}+1}}\limits_{k-j\equiv
m+1}}\frac{e^{-\alpha_{0}(\frac{j-k-1-m}{2})}}{\sqrt{\frac{j-k-1-m}{2}}}\frac{e^{-\alpha_{0}k}}{\sqrt
k}+\mathrel{\mathop{\sum^{j-1}_{k=\frac{j}{2}+1}}\limits_{k-j\equiv
m}}\frac{e^{-\alpha_{0}(\frac{j-k-m}{2})}}{\sqrt{\frac{j-k-m}{2}}}\frac{e^{-\alpha_{0}k}}{\sqrt
k} \right]\\
& \leq K\frac{e^{-\alpha_{0}\frac{j}{2}}}{\sqrt \frac{j}{2}}
\left[\mathrel{\mathop{\sum^{j-1}_{k=\frac{j}{2}+1}}\limits_{k-j\equiv
m+1}}
e^{-\alpha_{0}(\frac{k}{2}-m-1)}+\mathrel{\mathop{\sum^{j-1}_{k=\frac{j}{2}+1}}\limits_{k-j\equiv
m}}e^{\alpha_{0}(\frac{k}{2}-m)}\right]\\
& =o \left(\frac{e^{-\alpha_{0}\frac{j}{2}}}{\sqrt
\frac{j}{2}}\right).
\end{align*}

Finally,
\begin{align*}
|S_{4}|&=\left|\sum_{k=j}^{\infty}(-1)^{k}a_{k-j+1}c_{k}\right|\\
&\leq K \left[ \mathrel{\mathop{\sum^\infty_{k=j}}\limits_{k\equiv
m+1}}\frac{e^{-\alpha_{0}(
\frac{k-j+1-m}{2})}}{\sqrt{\frac{k-j+1-m}{2}}}\frac{e^{-\alpha_{0}k}}{\sqrt
k} +\mathrel{\mathop{\sum^\infty_{k=j}}\limits_{k\equiv
m}}\frac{e^{-\alpha_{0}(
\frac{k-j-m}{2})}}{\sqrt{\frac{k-j-m}{2}}}\frac{e^{-\alpha_{0}k}}{\sqrt
k}  \right]\\
 &\leq K\frac{e^{-\alpha_{0}\frac{j}{2}}}{\sqrt
\frac{j}{2}} \left[
\mathrel{\mathop{\sum^\infty_{k=j}}\limits_{k\equiv
m+1}}e^{-\alpha_{0}(\frac{3k}{2}-j-m+1)}+\mathrel{\mathop{\sum^\infty_{k=j}}\limits_{k\equiv
m+1}}e^{-\alpha_{0}(\frac{3k}{2}-j-m)} \right]\\
& =o \left(\frac{e^{-\alpha_{0}\frac{j}{2}}}{\sqrt
\frac{j}{2}}\right).\\
\end{align*}
\end{proof}

\begin{remark}
As mentioned in the Introduction, the result of Theorem 6.2 essentially gives us Theorem A. Similarly, Theorem 3.1 yields, for $m \geq 2,$
$$
\varphi_{m}(x)=K_c \left[\sum_{x-m \leq j \leq x} \frac{e^{- \alpha_0|j|}}{\sqrt{|j|}}N_{m}(x-j)\right][1+o(1)],
$$
as $|x|\rightarrow \rightarrow \infty.$
\\
\end{remark}

$$
$$

\end{document}